\documentclass[12pt]{article}
\usepackage{pict2e}
\usepackage{graphicx}
\usepackage{subfigure}
\usepackage{graphicx}
\usepackage{color}
\usepackage{latexsym}
\usepackage{multirow}
\usepackage{graphicx}
\usepackage{float}
\usepackage{xcolor}
\usepackage{mathrsfs}
\def\q{\hfill\rule{1ex}{1ex}}
\def\0{\emptyset}

\newcommand{\liuhao}{\fontsize{7.25pt}{\baselineskip}\selectfont}

\begin{document}
	\newtheorem{claim}{Claim}[section]
	\newtheorem{thm}{Theorem}[section]
	\newtheorem{cor}{Corollary}[section]
	\newtheorem{conj}{Conjecture}[section]
	\newtheorem{ques}{Question}[section]
	\newtheorem{lem}{Lemma}[section]
	\newtheorem{prop}{Proposition}[section]
	\newenvironment{proof}{\noindent {\bf
			Proof.}}{\rule{2mm}{2mm}\par\medskip}
	\newcommand{\remark}{\medskip\par\noindent {\bf Remark.~~}}
	\newcommand{\pp}{{\it p.}}
	\newcommand{\de}{\em}

	\title{\bf On the maximal Sombor index of quasi-tree graphs}

\author{{ 
		{\small\bf Ruiting ZHANG}\thanks{Email: deliazhangruiting@163.com; partially supported by NNSFC(No. 11971158).} {\quad
			{\small\bf Huiqing LIU}\thanks{Email: hqliu@hubu.edu.cn; partially supported by NNSFC (No. 11971158).}}{\quad{\small\bf Yibo LI}\thanks{Email: yiboli\_2019@163.com; partially supported by NNSFC (No. 11971158). Corresponding author.}}}\\ {\small Hubei Key Laboratory of Applied Mathematics, Faculty of
		Mathematics and Statistics, }\\{\small Hubei University, Wuhan 430062, China}}


\date{}
\maketitle \baselineskip 17.2pt

\begin{abstract}
	The Sombor index $SO(G)$ of a graph $G$ is the sum of the edge weights $\sqrt{d^2_G(u)+d^2_G(v)}$ of all edges $uv$ of $G$, where $d_G(u)$ denotes the degree of the vertex $u$ in $G$. A connected graph $G = (V ,E)$ is called a quasi-tree, if there exists $u\in V (G)$ such that $G-u$ is a tree. Denote $\mathscr{Q}(n,k)$=\{$G$: $G$ is a quasi-tree graph of order $n$ with $G-u$ being a tree and $d_G(u)=k$\}. In this paper, we determined the maximum, the second maximum and the third maximum Sombor indices of all quasi-tree graphs in $\mathscr{Q}(n,k)$, respectively. Moreover, we characterized their corresponding extremal graphs, respectively.
	
	
\end{abstract}





\section{Introduction}

Let $G=(V(G),E(G))$ be a simple undirected graph. For $v\in V(G)$, we use $N_G(v)$ to denote the \textit{set of  neighbors} of $v$ in $G$, and the \textit{degree} of $v$ is $d_G(v)=|N_G(v)|$. An \textit{$i$-vertex} is
a vertex of degree $i$. Let $V_i(G)$ be the set of all $i$-vertices in $G$. For a subgraph $H$ of $G$, let $N_H(v) = N_G(v) \cap V(H)$ and $d_H(v) = |N_H(v)|$ for $v \in V (H)$. We will use $G-v$ or $G-uv$ to denote the graph that arises from $G$ by deleting the vertex $v \in V(G)$ or the edge $uv \in E(G)$. Similarly, $G+uv$ is a graph that arises from $G$ by adding an edge $uv \notin E(G)$, where $u,~v \in V(G)$.

A \textit{tree} is a connected acyclic graph, and a  \textit{unicyclic graph} is a connected graph $G$ with  $|V(G)|=|E(G)|$. If there exists a vertex $u\in V(G)$ such that $G-u$ is a tree, then $G$ is called a \textit{quasi-tree}. Let $\mathscr{Q}(n,k)$=\{$G$: $G$ is a quasi-tree graph of order $n$ with $G-u$ being a tree and $d_G(u)=k$\}. Then $k\ge 1$ and $\mathscr{Q}(n,1)$ is the set of all trees of order $n$.

A vertex-degree-based-topological index was recently introduced by Gutman \cite{IGA8}, called the {\em Sombor index}, and defined for a graph $G$ as
$$SO(G)=\sum_{uv\in E(G)} \sqrt{d^2_G(u)+d^2_G(v)}. $$ Since then, the problem concerning graphs with the maximal or minimal Sombor index of a given class of graphs has been studied extensively, and many results had
been obtained (see \cite{HLWH2}-\cite{SO22}).

In \cite{KCIG3,IGA8}, Gutman presented some properties of the Sombor index and characterized the maximal and minimal graphs with respect to the Sombor index. Zhou {\em et al.} \cite{TZL4,SO22} obtained the maximum and minimum Sombor indices of trees and unicyclic graphs with given maximum degree or matching number, respectively. Recently, Das and Gutman \cite{KCIG3} gave the maximum Sombor index of all quasi-tree graphs, and also obtained the second maximum and the minimum extremal trees, respectively. In this paper, we determined the maximum, the second maximum and the third maximum Sombor indices of all quasi-tree graphs in the set $\mathscr{Q}(n,k)$, respectively, which generalized some known results.


\section{Lemmas}

In this section, we first give some lemmas that will be used in the proof of main results.

\begin{lem}
	\label{H}
	Let $g^{(r)}(x,y)=\sqrt{x^2+y^2}-\sqrt{x^2+(y-r)^2}$ with $r>0$. Then $g^{(r)}(x,y)$ is monotonic decreasing in $x\geq 1$ and monotonic increasing in $y\geq 1$, respectively.
\end{lem}

\begin{proof} Since $r>0$, $y^2>(y-r)^2$, we have
	\begin{eqnarray*}
		\frac{\partial g^{(r)}(x,y)}{\partial x}&=&\frac{x}{\sqrt{x^2+y^2}}-\frac{x}{\sqrt{x^2+(y-r)^2}}<0,
	\end{eqnarray*}	
	and
	$$\frac{\partial g^{(r)}(x,y)}{\partial y}=\frac{y} {\sqrt{x^2+y^2}}-\frac{y-r} {\sqrt{x^2+(y-r)^2}}=\frac{1} {\sqrt{1+\frac{x^2}{y^2}}}-\frac{1} {\sqrt{1+\frac{x^2}{(y-r)^2}}}>0. $$
	Thus the function $g^{(r)}(x,y)$ is monotonic decreasing in $x\geq 1$ and monotonic increasing in $y\geq 1$, respectively.
\end{proof}

\begin{lem}
	\label{x,y-r}	
	Let $h(x)=\frac{1}{\sqrt{x^2+a^2}}$. Then $\sigma(x):=h(x)-h(x-1)$ is monotonic decreasing in $0< x<\frac{\sqrt{2}}{2}a$.
\end{lem}

\begin{proof}
	Note that $\frac{d^2h(x)}{dx^2}=(x^2+a^2)^{-\frac{5}{2}}(2x^2-a^2)< 0$ for $0< x< \frac{\sqrt{2}}{2}a$, and hence $\frac{d\sigma(x)}{dx}<0$, which implies the function $\sigma(x)$ is monotonic decreasing in $0<x< \frac{\sqrt{2}}{2}a$.
\end{proof}	

\begin{lem}
	\label{G'>G}	
	Let $G$ be a graph, and let $uv,xy\in E(G)$ with $d_G(u)\ge d_G(x)$ and $d_G(v)\ge d_G(y)$. Set $G'=G-\{uv,xy\}+\{uy,xv\}$, then $$SO(G')\geq SO(G).$$ Moreover, equality holds if and only if $d_G(u)=d_G(x)$ or $d_G(v)=d_G(y)$.
\end{lem}

\begin{proof}
	Note that $V(G)=V(G')$ and $d_G(w)=d_{G'}(w)$ for all $w\in V(G)$. Then
	\begin{eqnarray*}
		&&SO(G')-SO(G)\\
		&=&\sqrt{d^2_G(x)+d^2_G(v)}+\sqrt{d^2_G(u)+d^2_G(y)}-\sqrt{d^2_G(x)+d^2_G(y)}-\sqrt{d^2_G(u)+d^2_G(v)}.
	\end{eqnarray*}	
	Clearly, $SO(G')= SO(G)$ if and only if $d_G(u)=d_G(x)$ or $d_G(v)=d_G(y)$. So we can assume that $d_G(v)> d_G(y)$ and $d_G(u)>d_G(x)$. Let $r:={d_G(v)-d_G(y)}>0$, then by Lemma \ref{H}, $$SO(G')-SO(G)=g^{(r)}(d_G(x),d_G(v))-g^{(r)}(d_G(u),d_G(v))> 0$$ as $d_G(u)>d_G(x)$. Therefore, the proof of Lemma \ref{G'>G} is complete.
\end{proof}	

For a graph $G$, define a relation $\succ^G$ on the vertex set $V(G)$ of $G$ as follows:

$x\succ^G z$ if and only if $d_G(x)\geq d_G(z)$ and $d_G(x')\leq d_G(z')$ for any $x'\in N_G(x)\setminus\{z\}$ and $z'\in N_G(z)\setminus\{x\}$.


\begin{lem}
	\label{G*>G}	
	Let $G$ be a graph, and let $xy\notin E(G)$ and $yz\in E(G)$. Set $G^*=G-{yz}+{xy}$.
	If $x\succ^G z$, then $$SO(G^*)> SO(G).$$
\end{lem}

\begin{proof} Let $N_G(x)=\{x_1,x_2,\ldots,x_t\}$, $N_G(z)=\{z_0,z_1,z_2,\ldots,z_s\}$, where $z_0=y$. Notice that  $d_{G^*}(x)=d_{G}(x)+1=t+1$, $d_{G}(z)=d_{G^*}(z)+1=s+1$, and $d_{G}(v)=d_{G^*}(v)$ for $v\in V(G)\setminus\{x,z\}$,
	then
	\begin{eqnarray*}
		~~~~~~SO(G^*)-SO(G)
		&=& \sqrt{(t+1)^2+d^2_G(z_0)}-\sqrt{(s+1)^2+d^2_G(z_0)}\\
		&&+\sum^t_{i=1 }\left(\sqrt{(t+1)^2+d^2_G(x_i)}-\sqrt{t^2+d^2_G(x_i)}\right)\\
		&&+\sum^s_{j=1}\left(\sqrt{s^2+d^2_G(z_j)}-\sqrt{(s+1)^2+d^2_G(z_j)}\right)\\
		&=&\sqrt{(t+1)^2+d^2_G(z_0)}-\sqrt{(s+1)^2+d^2_G(z_0)}\\
		&&+\sum^t_{i=1}g^{(1)}(d_G(x_i),t+1)-\sum^s_{j=1 }g^{(1)}(d_G(z_j),s+1).~~~~~ ~~~~(1)
	\end{eqnarray*}	
	If $x\succ^G z$, then $d_G(x)\geq d_G(z)$ (i.e., $t\geq s+1$) and $d_G(x_i)\leq d_G(z_j)$ for any $1\leq i\leq t$ and  $0\leq j\leq s$. Note that $t+1>s+1$, then by (1) and Lemma \ref{H},
	\begin{eqnarray*}
		SO(G^*)-SO(G)&>&\sum^t_{i=1}g^{(1)}(d_G(x_i),t+1)-\sum^s_{j=1 }g^{(1)}(d_G(z_j),s+1)\\
		&>&\sum^s_{i=1}g^{(1)}(d_G(x_i),t+1)-\sum^s_{j=1 }g^{(1)}(d_G(z_j),s+1)\\\
		&\ge &\sum^s_{j=1}\left[g^{(1)}(d_G(z_j),t+1)-g^{(1)}(d_G(z_j),s+1)\right]>0.
	\end{eqnarray*}	
	Hence, the proof of Lemma \ref{G*>G} is complete.
\end{proof}	

\section{Results}

In this section, we will determine the maximum, the second maximum and the third maximum Sombor indices of all quasi-tree graphs in $\mathscr Q(n,k)$. In order to formulate our results, we need to define some graphs (see Figure 1) as follows.

Let $S_n\cong K_{1,n-1}$, a {\em star} of order $n$, and let $S'_{n}$ (see Table 1) be a graph of order $n$ obtained from $S_{n-1}$ by attaching an isolated vertex to one $1$-vertex of $S_{n-1}$. Let $S''_{n}$ (see Table 1) be a graph of order $n$ obtained from $S'_{n-1}$ by attaching an isolated vertex to the $2$-vertex of $S'_{n-1}$.

Let $Q_{n,k}$ be a graph obtained from a star $S_{n-1}$ and an isolated vertex $v$ by an edge joining $v$ and the only $(n-2)$-vertex of $S_{n-1}$,  and adding $k-1$ edges joining $v$ and the $1$-vertices of $S_{n-1}$, respectively.

Let $Q^*_{n,k}$ be a graph obtained from a star $S_{n-1}$ and an isolated vertex $v$ by adding $k$ edges joining $v$ and the $1$-vertices of $S_{n-1}$. Let $Q'_{n,1}\cong S'_{n}$, and let $Q'_{n,2}$ be a graph obtained from $Q_{n-1,2}$ by attaching an isolated vertex to one $2$-vertex in $Q_{n-1,2}$. Let $Q'_{n,k}$ ($3\le k\le n-1$) be a graph obtained from $Q_{n-1,k-1}$ by attaching an isolated vertex to one $2$-vertex and the $(k-1)$-vertex in $Q_{n-1,k-1}$.

Let $Q''_{n,1}\cong S''_{n}$, and let $Q''_{n,3}$ (see Table 1) be a graph obtained from $Q_{n-1,3}$ by attaching an isolated vertex to one $2$-vertex in $Q_{n-1,3}$. Let $Q''_{n,n-1}$ be a graph obtained from $Q'_{n-1,n-2}$ by attaching an isolated vertex to one $3$-vertex and the $(n-2)$-vertex in $Q'_{n-1,n-2}$.

Then $Q_{n,k}, Q^*_{n,k}, Q'_{n,k}\in \mathscr Q(n,k)$, $Q_{n,1}\cong
S_{n}$ and $Q^*_{n,1}\cong S'_{n}$.

\begin{center}
	\setlength{\unitlength}{1mm}
	\begin{picture}(100,50)
		\thinlines
		
		\put(0,35){\circle*{1.5}} \put(-10,25){\circle*{1.5}}
		\put(-7.5,25){\circle*{0.8}}
		\put(-6,25){\circle*{0.8}}
		\put(-4.5,25){\circle*{0.8}}
		\put(-2,25){\circle*{1.5}} \put(3,25){\circle*{1.5}}
		\put(5,25){\circle*{0.8}}
		\put(6.5,25){\circle*{0.8}}
		\put(8,25){\circle*{0.8}}
		\put(10,25){{\circle*{1.5}}}
		\put(10,45){{\circle*{1.5}}}
		
		\put(0,35){\line(1,-1){10}}
		\put(0,35){\line(-1,-1){10}}
		\put(0,35){\line(-2,-10){2}}
		\put(0,35){\line(3,-10){3}}
		\put(0,35){{\color{red}{\line(1,1){10}}}}
		\put(10,45){{\color{red}{\line(-7,-20){7}}}}
		\put(10,45){{\color{red}{\line(0,-1){20}}}}
		
		\put(-3,38){$v_1$} \put(9,47){$v$}
		\put(-11,24){$\underbrace{~~~~~~}_{n-k-1}$} \put(29,24){$\underbrace{~~~~~~}_{n-k-2}$}
		\put(3,24){$\underbrace{~~~~~~}_{k-1}$} \put(68,24){$\underbrace{~~~~~~}_{n-k-1}$}
		
		\put(-3,12){$Q_{n,k}$}

		\put(40,35){\circle*{1.5}} \put(30,25){\circle*{1.5}}
		\put(32.5,25){\circle*{0.8}}
		\put(34,25){\circle*{0.8}}
		\put(35.5,25){\circle*{0.8}}
		\put(38,25){\circle*{1.5}} \put(43,25){\circle*{1.5}}
		\put(45,25){\circle*{0.8}}
		\put(46.5,25){\circle*{0.8}}
		\put(48,25){\circle*{0.8}}
		\put(50,25){{\circle*{1.5}}}
		\put(50,45){{\circle*{1.5}}}
		
		\put(40,35){\line(1,-1){10}}
		\put(40,35){\line(-1,-1){10}}
		\put(40,35){\line(-2,-10){2}}
		\put(40,35){\line(3,-10){3}}
		\put(50,45){{\color{red}{\line(-7,-20){7}}}}
		\put(50,45){{\color{red}{\line(0,-1){20}}}}
		
		\put(38,38){$v_1$} \put(49,47){$v$}
		\put(43,24){$\underbrace{~~~~~~}_{k}$}
		
		\put(37,12){$Q^*_{n,k}$}
		
		\put(80,35){\circle*{1.5}} \put(70,25){\circle*{1.5}}
		\put(71.5,25){\circle*{0.8}}
		\put(73,25){\circle*{0.8}}
		\put(74.5,25){\circle*{0.8}}
		\put(76,25){\circle*{1.5}}
		\put(79,25){\circle*{1.5}}
		\put(80.5,25){\circle*{0.8}}
		\put(82,25){\circle*{0.8}}
		\put(83.5,25){\circle*{0.8}}
		\put(85,25){\circle*{1.5}}
		\put(90,25){{\circle*{1.5}}}
		\put(90,45){{\circle*{1.5}}}
		\put(100,25){{\circle*{1.5}}}
		
		\put(80,35){\line(1,-1){10}}
		\put(80,35){\line(-1,-1){10}}
		\put(80,35){\line(-4,-10){4}}
		\put(80,35){\line(-1,-10){1}}
		\put(80,35){\line(5,-10){5}}
		\put(90,25){\line(1,0){10}}
		\put(80,35){{\color{red}{\line(1,1){10}}}}
		\put(79,24){$\underbrace{~~~~~}_{k-3}$}
		\put(90,45){{\color{red}{\line(0,-1){20}}}}
		\put(90,45){{\color{red}{\line(-11,-20){11}}}}
		\put(90,45){{\color{red}{\line(-5,-20){5}}}}
		\put(90,45){{\color{red}{\line(--10,-20){10}}}}
		
		\put(78,38){$v_1$}
		\put(89,20){$v_2$}
		\put(99,20){$v_3$}
		\put(89,47){$v$}
		
		\put(77,12){$Q'_{n,k}$ ($k\ge 3$)}

		\put(22,2){Figure 1.~Some graphs in $\mathscr{Q}(n,k)$ }
		
	\end{picture}
\end{center}

Denote $ \phi(n,k):=(n-k-1)\sqrt{(n-1)^2+1}+(k-1)\left(\sqrt{k^2+4}+\sqrt{(n-1)^2+4}\right)$ $+\sqrt{k^2+(n-1)^2}$ for $1\le k\le n-1$.

\begin{thm}
	\label{max}
	Let $G \in \mathscr{Q}(n,k)$ with $1\le k \leq n-1$. Then
	$$SO(G) \leq  \phi(n,k)$$ with equality holds if and only if $G \cong Q_{n,k}.$
\end{thm}

\begin{proof}
	First we note that if $G\cong Q_{n,k}$, then $SO(G)=SO(Q_{n,k})= \phi(n,k)$. Now we will show that if $G \in \mathscr{Q}(n,k)$ with $k \geq 1$, then $SO(G) \leq SO(Q_{n,k})$ and equality holds only if $G \cong Q_{n,k}$.
	Choose $G \in \mathscr{Q}(n,k)$ such that
	
	(C-1)~~ $SO(G)$ is as large as possible.
	
	Let $V(G)=\{v_1,\ldots,v_{n-1},v_n\}$. If $d_{G}(v_i)\ge d_{G}(v_j)$, then by Lemma \ref{G'>G} and (C-1), we can assume that $v_i\succ^{G} v_j$.
	Assume that $G-v_n$ is a tree. Denote $T=G-v_n$. Choose $v_1\in V(T)$ such that
	
	(C-2)~~ subject to (C-1), $d_{G}(v_1)$ is as large as possible.
	
	We will first show that $v_1v_n\in E(G)$. Otherwise, since $k\ge 1$, there exists some $i,~2\le i\le n-1$ such that $v_nv_i\in E(G)$. By (C-2), $d_{G}(v_1)\ge d_{G}(v_i)$, that is $v_1\succ^{G} v_i$. Let $G':=G-v_nv_i+v_nv_1$. Then $G' \in \mathscr{Q}(n,k)$. By Lemma \ref{G*>G}, $SO(G')>SO(G)$, a contradiction with (C-1).
	
	Next we will show that $v_1$ is adjacent to each vertex of $T-v_1.$ Otherwise, let $v_j\in V(T-v_1)$ with $v_1v_j\notin E(G)$, and let $v_1,v_2,v_3$ the first three vertices on the unique ($v_1$,$v_j$)-path in $T$. Set $G^*:=G-v_2v_3+v_1v_3$. Then $G^*\in \mathscr{Q}(n,k)$. Note that $v_1\succ^{G} v_2$, and hence, by Lemma \ref{G*>G}, $SO(G^*)>SO(G)$, a contradiction.
	
	Therefore $T$ is a star as $T$ is a tree, that is $G \cong Q_{n,k}$.
\end{proof}

Denote

{\small
	\noindent $~~~~ \phi'(n,1)=(n-3)\sqrt{(n-2)^2+1}+\sqrt{(n-2)^2+4}+\sqrt{5}$,
	
	\noindent $~~~~ \phi'(n,2)=(n-4)\sqrt{(n-2)^2+1}+\sqrt{(n-2)^2+4}+\sqrt{(n-2)^2+9}+\sqrt{13}+\sqrt{10}$,
	
	
	\noindent $~~~~ \phi'(n,k)=(k-2)\sqrt{k^2+4}+(k-3)\sqrt{(n-2)^2+4}+(n-k-1)\sqrt{(n-2)^2+1}\\~~~~~~~~~~~~~~~~~~~+\sqrt{k^2+(n-2)^2}+\sqrt{(n-2)^2+9}+\sqrt{k^2+9}+\sqrt{13}$~~ for $3\leq k\leq n-1$,
	
	\noindent $~~~~ \phi^*(n,k)=k(\sqrt{k^2+4}+\sqrt{(n-2)^2+4})+(n-k-2)\sqrt{(n-2)^2+1}$ for $1\leq k\leq n-2$.}

\vskip .2cm

The following inequalities are useful in the proof of another result.

\begin{prop} \label{1}
	For $1\le k\le n-2$, we have
	
	(i) $ \phi'(n,k)>  \phi^*(n,k)$ for $2\le k\le 3$;
	
	(ii) $ \phi'(n,k)<  \phi^*(n,k)$ for $4\le k\leq n-2$ and $n\geq 23$.
\end{prop}

\begin{proof}	
	Let $f(n,k)= \phi^*(n,k)-  \phi'(n,k)$. Then
	{\small	\begin{eqnarray*}
			f(n,2)&=&\sqrt{(n-2)^2+4}-\sqrt{(n-2)^2+9}+2\sqrt{8}-\sqrt{13}-\sqrt{10}<0,\\
			f(n,3)&=&3\left(\sqrt{3^2+4}+\sqrt{(n-2)^2+4}\right)+(n-5)\sqrt{(n-2)^2+1}\\
			&&-(n-4)\sqrt{(n-2)^2+1}-2\sqrt{(n-2)^2+9}-\sqrt{18}-2\sqrt{13}\\
			&<&2\sqrt{(n-2)^2+4}-\sqrt{(n-2)^2+1}-\sqrt{(n-2)^2+9}\\
			&=&g^{(1)}{(n-2,2)}-g^{(1)}{(n-2,3)}<0,\\
			f(n,k)
			&=&2\sqrt{k^2+4}+3\sqrt{(n-2)^2+4}-\sqrt{(n-2)^2+k^2}-\sqrt{(n-2)^2+9}\\
			&&-\sqrt{(n-2)^2+1}-\sqrt{k^2+9}-\sqrt{13}
	\end{eqnarray*}}	 for $4\le k\leq n-2$. In particular, $f(23,4)
	=2\sqrt{20}+3\sqrt{21^2+4}-\sqrt{21^2+4^2}-\sqrt{21^2+9}
	-\sqrt{21^2+1}-5-\sqrt{13}\approx 0.0092>0.$	
	
	If $k\geq 4$ and $n\geq 23$, then by Lemma \ref{x,y-r},
	{\small \begin{eqnarray*}
			\frac{\partial f(n,k)}{\partial n}&=&\frac{3(n-2)}{\sqrt{(n-2)^2+2^2}}-\frac{n-2}{\sqrt{(n-2)^2+3^2}}-\frac{n-2}{\sqrt{(n-2)^2+1^2}}
			-\frac{n-2}{\sqrt{(n-2)^2+k^2}}\\
			&>& \frac{2(n-2)}{\sqrt{(n-2)^2+2^2}}-\frac{n-2}{\sqrt{(n-2)^2+3^2}}-\frac{n-2}{\sqrt{(n-2)^2+1^2}}\\
			&=&(n-2)(\sigma(2)-\sigma(3))>0,
	\end{eqnarray*}}
	which implies $f(n,k)$ is monotonic increasing in $n\ge 23$. On the other hand, since
	{\small\begin{eqnarray*}
			\frac{\partial f(n,k)}{\partial k}&=&\frac{2k}{\sqrt{k^2+2^2}}-\frac{k}{\sqrt{(n-2)^2+k^2}}-\frac{k}{\sqrt{k^2+3^2}}> 0,
	\end{eqnarray*}	} we have
	$f(n,k)$ is strictly monotonic increasing in $k\ge 1$. Therefore, $f(n,k)\ge f(23,k)\ge f(23,4)>0$ for $k\geq 4$ and $n\geq 23$.
\end{proof}

{\bf Note.} $ \phi'(n,k)$ and $ \phi^*(n,k)$ are incomparable for $k\geq 4$ and $n\leq 22$. For example, $ \phi'(9,5)> \phi^*(9,5)$ and $ \phi'(22,5)< \phi^*(22,5)$ as
\begin{eqnarray*}
	f(9,5)
	&=&2\sqrt{29}+3\sqrt{53}-\sqrt{74}-\sqrt{58}
	-\sqrt{50}-\sqrt{34}-\sqrt{13}\approx-0.1150<0,\\
	f(22,5)
	&=&2\sqrt{29}+3\sqrt{404}-\sqrt{425}-\sqrt{409}
	-\sqrt{401}-\sqrt{34}-\sqrt{13}\approx 0.7688>0.
\end{eqnarray*}	

\begin{thm}\label{max2}
	Let $G \in \mathscr{Q}(n,k)\setminus \{Q_{n,k}\}$ with $n\ge 5$ and $1\le k \leq n-1$. Then
	$$~~~~~~~~~~~~~~~~~~SO(G)\leq\left\{
	\begin{array}{ll}
		\phi'(n,k),      & {if~ k=1,2,3,n-1};~~~~~~~~~~~~~~~~~~~~~~~~~~~~~~~(2)\\
		\phi^*(n,k),     & {if~ 4\le k\leq n-2 ~and ~n\geq 23}.~~~~~~~~~~~~~~~~~~~(3)	
	\end{array} \right. $$
	Moreover, equality holds in (2) (or resp. (3)) if and only if $G\cong Q'_{n,k}$ (or resp. $G\cong Q^*_{n,k}$).
\end{thm}

\begin{proof}
	First we note if $G\cong Q'_{n,k}$ (or resp. $G\cong Q^*_{n,k}$), then $SO(G)=SO(Q'_{n,k})= \phi'(n,k)$ (or resp. $SO(G)=SO(Q^*_{n,k})= \phi^*(n,k)$).
	
	Now we have to prove that if $G \in \mathscr{Q}(n,k)\setminus \{Q_{n,k}\}$, then (2) (or resp. (3)) holds and  equality in (2) (or resp. (3)) holds only if $G\cong Q'_{n,k}$ (or resp. $G\cong Q^*_{n,k}$). Choose $G \in \mathscr{Q}(n,k)\setminus \{Q_{n,k}\}$ such that
	
	(D-1)~~ $SO(G)$ is as large as possible.
	
	Let $V(G)=\{v_1,\ldots,v_{n-1},v_n\}$. If $d_G(v_i)\ge d_G(v_j)$, then we can assume that $v_i\succ^G v_j$ by Lemma \ref{G'>G} and (D-1). Suppose, without loss of generality, that $G-v_n$ is a tree. Set $T:=G-v_n$. Choose $v_1\in V(T)$ such that
	
	(D-2)~~subject to (D-1), $d_G(v_1)$ is as large as possible.
	
	We consider the following two cases.
	
	\vskip .2cm	
	{\bf Case 1.} $T$ is a star.
	
	\vskip .2cm	
	In this case, $v_1v_n\notin E(G)$ as $G \not\cong Q(n,k)$. Then $k\leq n-2$ and $G\cong Q^*(n,k)$ and thus the assertion holds for $4\le k\le n-2$. Note that $ Q^*(n,1)\cong Q'(n,1)$, and hence the assertion holds for $k=1$. If $2\le k\leq 3$, then by Proposition \ref{1}, $SO(G)=  \phi^*(n,k)<  \phi'(n,k)=SO(Q'_{n,k})$.
	
	\vskip .2cm	
	
	{\bf Case 2.} $T$ is not a star.
	
	\vskip .2cm
	
	In this case, we will first show that  $v_1v_n\in E(G)$. Otherwise, there exists some $i$, $2\leq i \leq n-1$ such that $v_nv_i\in E(G)$ as $k\geq 1$. By (D-2), $d_G(v_1)\geq d_G(v_i)$, that is $v_1\succ^G v_i$. Let  $G':=G-v_nv_i+v_nv_1$, then $G' \in \mathscr{Q}(n,k)\setminus \{Q_{n,k}\}$. By Lemma \ref{G*>G}, we can get that  $SO(G')>SO(G)$, a contradiction with (D-1). So $v_1v_n\in E(G)$. Next we will show some facts.
	
	\vskip .2cm
	
	{\bf Fact 1.} $T\cong S'_{n-1}$.
	
	\vskip .2cm
	
	{\bf Proof of Fact 1.} First we claim that $d_T(v_1,v_i)\le 2$ for $v_i\in V(T)$. Otherwise, there exists some vertex $v_i\in V(T)$ such that $d_T(v_1,v_i)\geq 3$. Let $P':=v_1v_2v_3v_4\cdots v_i$ be the unique $(v_1,v_i)$-path (possibly $v_4=v_i$). By the choice of $v_1$, $d_G(v_1)\geq d_G(v_3)$, i.e., $v_1\succ^G v_3$. Set $G'=G-v_3v_4+v_1v_4$. Then $G'\in \mathscr{Q}(n,k)\setminus \{Q_{n,k}\}$. By Lemma \ref{G*>G}, $SO(G')>SO(G)$, a contradiction. So $d_T(v_1,v_i)\le 2$ for $v_i\in V(T)$.
	
	Since $T$ is not a star, there exists some $v_i\in V(T)$ such that $d_T(v_1,v_i)=2$. Let $v_1v_2v_i$ be the $(v_1,v_i)$-path, and $V'=\{v_i \in V(T)~|~d_T(v_1,v_i)=2\}$. Then $V'=\{v_i\}$. Otherwise, set $G''=G-v_2v_i+v_1v_i$. Then $G''\in \mathscr{Q}(n,k)\setminus \{Q_{n,k}\}$. Note that $v_1\succ^G v_2$, and hence, by Lemma \ref{G*>G}, $SO(G'')>SO(G)$, a contradiction. So $T \cong S'_{n-1}$.\q
	
	\vskip .2cm
	
	By Fact 1, we let $v_2$ be the only vertex of degree 2 and $v_3$ the only vertex being nonadjacent to $v_1$ in $T$. If $k=1$, then $G\cong Q'_{n,1}$, and the assertion holds for $k=1$. So, in the following, we assume that $k \geq 2$.
	
	\vskip .2cm	
	{\bf Fact 2.} {\em $v_nv_2\in E(G)$.}
	
	\vskip .2cm
	
	{\bf Proof of Fact 2.} If $v_nv_2\notin E(G)$, then $k\le n-2$ and there exists a vertex $v_j\in V(T)\setminus \{v_1,v_2\}$ such that $v_nv_j\in E(G)$ as $k\geq 2$. Set $G'=G-v_nv_j+v_nv_2$. Then $G'\in \mathscr{Q}(n,k)\setminus \{Q_{n, k}\}$. If $j=3$, then
	$SO(G')-SO(G)=\sqrt{(n-2)^2+3^2}+\sqrt{k^2+3^2}+\sqrt{10}-\sqrt{(n-2)^2+2^2}-\sqrt{k^2+2^2}-\sqrt{8}
	>0,$
	and if $j\neq 3$, then by Lemma \ref{H},
	\begin{eqnarray*}
		SO(G')-SO(G)
		&=&\sqrt{(n-2)^2+3^2}+\sqrt{(n-2)^2+1^2}+\sqrt{k^2+3^2}+\sqrt{1^2+3^2}\\
		&& -2\sqrt{(n-2)^2+2^2}-\sqrt{k^2+2^2}-\sqrt{1^2+2^2}\\
		&=&g^{(1)}(n-2,3)-g^{(1)}(n-2,2)+g^{(1)}(k,3)+g^{(1)}(1,3)>0,
	\end{eqnarray*}
	a contradiction. Therefore $v_nv_2\in E(G)$.\q
	
	\vskip .2cm
	
	{\bf Fact 3.} {\em If $k\ge 3$, then $v_nv_3\in E(G)$.}
	
	\vskip .2cm
	
	{\bf Proof of Fact 3.} Assume that $v_nv_3\notin E(G)$. Then $k\le n-2$. Since $k\geq 3$, there exists a vertex $v_4\in V(T)\setminus \{v_1,v_2,v_3\}$ such that $v_nv_4\in E(G)$. If $n=5$, then $k=3$ and $G\cong Q'_{n,3}$, and thus we can assume that $n\ge 6$. Set $G''=G-v_nv_4+v_nv_3$. Then $G''\in \mathscr{Q}(n,k)\setminus \{Q_{n,k}\}$. By Lemma \ref{H},
	\begin{eqnarray*}
		SO(G'')-SO(G)
		&=&\sqrt{(n-2)^2+1^2}+\sqrt{2^2+3^2}-\sqrt{(n-2)^2+2^2}-\sqrt{1^2+3^2}\\
		&=&g^{(1)}(3,2)-g^{(1)}(n-2,2)>0 	
	\end{eqnarray*} as $n-2>3$, a contradiction. Therefore $v_nv_3\in E(G)$.\q
	
	By Facts 2 and 3, $G\cong Q'_{n,k}$, thus the assertion holds for $k=2,3,n-1$. If $4\le k\le n-2$ and $n\ge 23$, then by Proposition \ref{1}, $SO(G)= \phi'(n,k)< \phi^*(n,k)$. Therefore the proof of Theorem \ref{max2} is completed.
\end{proof}

From the proof of Theorem \ref{max2}, we have the following result.

\begin{thm}\label{max3-1}
	Let $G \in \mathscr{Q}(n,k)\setminus \{Q_{n,k},Q^*_{n,k}\}$ with $4\le k\leq n-2$ and $n\geq 23$. Then
	$$SO(G)\leq  \phi'(n,k)$$
	with equality holds if and only if $G\cong Q'_{n,k}$.
\end{thm}

Denote

{\small
	$\phi''(n,1)=(n-4)\sqrt{(n-3)^2+1}+\sqrt{(n-3)^2+9}+2\sqrt{10}$,
	
	
	$\phi''(n,3)=(n-5)\sqrt{(n-2)^2+1}+\sqrt{(n-2)^2+4}+2\sqrt{(n-2)^2+9}+\sqrt{10}+\sqrt{13}+\sqrt{18}$,
	
	$\phi''(n,n-1)=(n-3)\sqrt{(n-1)^2+4}+\sqrt{(n-1)^2+16}+\sqrt{(n-1)^2+(n-3)^2}\\~~~~~~~~~~~~~~~~~~~~~~~~~+\sqrt{(n-3)^2+16}+(n-5)\sqrt{(n-3)^2+4}+2\sqrt{20}$.}

\begin{prop} \label{2}
	For $n\ge 7$, we have
	
	(i) $(n-5)\sqrt{(n-3)^2+1}+\sqrt{(n-3)^2+4}+\sqrt{(n-3)^2+16}+\sqrt{20}+2\sqrt{17}< \phi^*(n,2)$;
	
	(ii) $(n-5)\sqrt{(n-3)^2+1}+\sqrt{(n-3)^2+9}+\sqrt{(n-3)^2+16}+\sqrt{20}+\sqrt{17}+\sqrt{13}+5<\phi''(n,3)$.
\end{prop}

\begin{proof}
	(i) Note that
	{\begin{eqnarray*}
			~~~~ &&(n-5)\sqrt{(n-3)^2+1}+\sqrt{(n-3)^2+4}+\sqrt{(n-3)^2+16}+\sqrt{20}+2\sqrt{17}\\
			&=& \phi^*(n,2)-2\sqrt{8}-2\sqrt{(n-2)^2+4}-(n-4)\sqrt{(n-2)^2+1}\\
			&&+(n-5)\sqrt{(n-3)^2+1}+\sqrt{(n-3)^2+4}+\sqrt{(n-3)^2+16}+2\sqrt{17}+\sqrt{20}\\
			&=&\phi^*(n,2)+(n-5)(\sqrt{(n-3)^2+1}-\sqrt{(n-2)^2+1})-2\sqrt{8}+2\sqrt{17}+\sqrt{20}\\
			&&+\sqrt{(n-3)^2+4}+\sqrt{(n-3)^2+16}-2\sqrt{(n-2)^2+4}-\sqrt{(n-2)^2+1}.~~~~~(*)
	\end{eqnarray*}	}
	Let $f_1(n)=\sqrt{(n-3)^2+4}+\sqrt{(n-3)^2+16}-2\sqrt{(n-2)^2+4}-\sqrt{(n-2)^2+1}$. Then
	$$\frac{df_1(n)}{dn}=\frac{1} {\sqrt{1+(\frac{2}{n-3})^2}}+\frac{1} {\sqrt{1+(\frac{4}{n-3})^2}}-\frac{2} {\sqrt{1+(\frac{2}{n-2})^2}}-\frac{1} {\sqrt{1+\frac{1}{(n-2)^2}}}<0,$$ which implies $f_1(n)\le \sqrt{20}+\sqrt{32}-2\sqrt{29}-\sqrt{26}$ as $n\ge 7$. Thus, by ($*$) and Lemma \ref{H},
	\begin{eqnarray*}
		~~~~ &&(n-5)\sqrt{(n-3)^2+1}+\sqrt{(n-3)^2+4}+\sqrt{(n-3)^2+16}+\sqrt{20}+2\sqrt{17}\\
		&=&\phi^*(n,2)-(n-5)g^{(1)}(1,n-2)+f_1(n)-2\sqrt{8}+2\sqrt{17}+\sqrt{20}\\
		&\le &\phi^*(n,2)+4\sqrt{17}+2\sqrt{20}+\sqrt{32}-3\sqrt{26}-2\sqrt{8}-2\sqrt{29}\\
		&=&\phi^*(n,2)-0.6307<\phi^*(n,2).
	\end{eqnarray*}
	
	(ii) Let $f_2(n):=\sqrt{(n-3)^2+9}+\sqrt{(n-3)^2+16}-2\sqrt{(n-2)^2+9}-\sqrt{(n-2)^2+4}$. Then $\frac{df_2(n)}{dn}=\frac{1} {\sqrt{1+(\frac{3}{n-3})^2}}+\frac{1} {\sqrt{1+(\frac{4}{n-3})^2}}-\frac{2} {\sqrt{1+(\frac{3}{n-2})^2}}-\frac{1} {\sqrt{1+\frac{2}{(n-2)^2}}}<0$,  and thus, $f_2(n)\le f_2(7)=5+\sqrt{32}-2\sqrt{34}-\sqrt{29}$ as $n\ge 7$. By Lemma \ref{H},
	{\small\begin{eqnarray*}
			&&(n-5)\sqrt{(n-3)^2+1}+\sqrt{(n-3)^2+9}+\sqrt{(n-3)^2+16}+\sqrt{20}+\sqrt{17}+\sqrt{13}+5\\
			&=&\phi''(n,3)-(n-5)(\sqrt{(n-2)^2+1}-\sqrt{(n-3)^2+1})+\sqrt{20}+\sqrt{17}+5\\
			&&+\sqrt{(n-3)^2+9}-2\sqrt{(n-2)^2+9}+\sqrt{(n-3)^2+16}-\sqrt{(n-2)^2+4}-\sqrt{10}-\sqrt{18}\\
			&=&\phi''(n,3)-(n-5)g^{(1)}(1,n-2)+f_2(n)+\sqrt{17}+\sqrt{20}+5-\sqrt{10}-\sqrt{18}\\
			&\le &\phi''(n,3)-2g^{(1)}(1,5)+10+\sqrt{32}-2\sqrt{34}-\sqrt{29}+\sqrt{17}+\sqrt{20}-\sqrt{10}-\sqrt{18}\\
			&=&\phi''(n,3)+10+\sqrt{32}+\sqrt{20}+3\sqrt{17}-2\sqrt{34}-\sqrt{29}-2\sqrt{26}-\sqrt{18}-\sqrt{10}\\
			&=&\phi''(n,3)-2.1517<\phi''(n,3).
	\end{eqnarray*}	}
	Therefore, the proof is complete.
\end{proof}

\begin{thm}\label{max3-2}
	Let $G \in \mathscr{Q}(n,k)\setminus \{Q_{n,k},Q'_{n,k}\}$ with $k\in\{1,2,3,n-1\}$ and $n\ge 7$. Then $$~~~~~~~~~~~~~~~~~~SO(G)\leq\left\{
	\begin{array}{ll}
		\phi''(n,k),      & {if~ k=1,3,n-1};~~~~~~~~~~~~~~~~~~~~~~~~~~~~~~(4)\\
		\phi^*(n,k),     & {if~ k=2}.~~~~~~~~~~~~~~~~~~~~~~~~~~~~~~~~~~~~~~~~~(5)	
	\end{array} \right. $$
	Moreover, equality holds in (4) (or resp. (5)) if and only if $G\cong Q''_{n,k}$ for $k\in\{1,3,n-1\}$ (or resp. $G\cong Q^*_{n,k}$ for $k=2$).
\end{thm}

\begin{proof} It suffices to show that if $G \in \mathscr{Q}(n,k)\setminus \{Q_{n,k},Q'_{n,k}\}$, then (4) (or resp. (5)) holds and  equality in (4) (or resp. (5)) holds only if $G\cong Q''_{n,k}$  for $k\in\{1,3,n-1\}$ (or resp. $G\cong Q^*_{n,k}$ for $k=2$). Choose $G \in \mathscr{Q}(n,k)\setminus \{Q_{n,k},Q'_{n,k}\}$ such that
	
	(E-1)~~ $SO(G)$ is as large as possible.
	
	Let $V(G)=\{v_1,\ldots,v_{n-1},v_n\}$, and let $T:=G-v_n$ a tree with $d_G(v_n)=k$. Choose $v_1\in V(T)$ such that	
	
	(E-2)~~subject to (E-1), $d_G(v_1)$ is as large as possible.
	
	Note that if $d_G(v_i)\ge d_G(v_j)$, then we can assume that $v_i\succ^G v_j$ by Lemma \ref{G'>G} and (E-1).
	We consider three cases.
	
	\vskip .2cm	
	
	{\bf Case 1.} $T\cong S_{n-1}$.
	
	\vskip .2cm
	
	In this case, $v_1$ is the $(n-2)$-vertex, $2\le k\le n-2$ and $G\cong Q^*_{n,k}$ as $G \not\cong Q_{n,k}, Q'_{n,k}$. So the assertion holds for $k=2$. If $k=3$, then
	\begin{eqnarray*}
		SO(G)
		&=&\phi^*(n,3)=3\sqrt{(n-2)^2+4}+3\sqrt{13}+(n-5)\sqrt{(n-2)^2+1}\\
		&=&\phi''(n,3)+2\sqrt{(n-2)^2+4}-2\sqrt{(n-2)^2+9}+2\sqrt{13}-\sqrt{18}-\sqrt{10}\\
		&<&\phi''(n,3).
	\end{eqnarray*} So the assertion holds.
	
	\vskip .2cm	
	
	{\bf Case 2.} $T\cong S'_{n-1}$.
	
	\vskip .2cm
	
	In this case, $v_1$ is the $(n-3)$-vertex in $T$. Let $v_2$ be the $2$-vertex in $T$ and $v_3$ the $1$-vertex adjacent to $v_2$ in $T$, then $k\not= n-1$ as $G \not\cong Q'_{n,k}$. We consider two subcases.

	\vskip .2cm	
	
	{\bf Subcase 2.1} $v_nv_1\in E(G)$.
	
	\vskip .2cm
	
	In this subcase, $k\ge 2$ as $G \not\cong Q'_{n,k}$. If $v_2v_n\in E(G)$, then $k\ge 3$ and $v_3v_n\notin E(G)$ as $G \not\cong Q'_{n,k}$, which implies $G\cong Q''_{n,3}$, the assertion holds.
	
	If $v_2v_n\notin E(G)$, then $k\neq n-1$. By an argument similar to the proof of Fact 3, we have $v_nv_3\in E(G)$. If $k=2$, then $G\cong Q^*_{n,2}$, the assertion holds. If $k=3$, then there exists some $v_4\in V(T)\setminus\{v_1,v_2,v_3\}$ such that $v_nv_4\in E(G)$, and then
	\begin{eqnarray*}
		SO(G)
		&=&(n-5)\sqrt{(n-2)^2+1}+2\sqrt{(n-2)^2+4}+\sqrt{(n-2)^2+9}+2\sqrt{13}+\sqrt{8}\\
		&=&\phi''(n,3)+\sqrt{(n-2)^2+4}+\sqrt{13}+\sqrt{8}-\sqrt{(n-2)^2+9}-\sqrt{10}-\sqrt{18}\\
		&<&\phi''(n,3).
	\end{eqnarray*}
	
	\vskip .2cm	
	
	{\bf Subcase 2.2} $v_nv_1\notin E(G)$.
	
	\vskip .2cm
	
	In this subcase, $k\neq n-1$ and by an argument similar to the proof of Fact 2, we have $v_nv_2\in E(G)$. If $k=1$, then $G\cong S_n''$, the assertion holds for $k=1$. So we can assume that $k\ge 2$. Then by an argument similar to the proof of Fact 3, we have $v_nv_3\in E(G)$. If $k=2$, then
	\begin{eqnarray*}
		SO(G)
		&=&(n-4)\sqrt{(n-3)^2+1}+\sqrt{(n-3)^2+9}+2\sqrt{13}+\sqrt{8}\\
		&=&\phi^*(n,2)+(n-4)(\sqrt{(n-3)^2+1}-\sqrt{(n-2)^2+1})\\
		&&+\sqrt{(n-3)^2+9}-2\sqrt{(n-2)^2+4}+2\sqrt{13}-\sqrt{8}\\
		&<&\phi^*(n,2)+2\sqrt{13}-\sqrt{(n-2)^2+4}-\sqrt{8}<\phi^*(n,2).
	\end{eqnarray*}
	If $k=3$, then
	\begin{eqnarray*}
		SO(G)
		&=&(n-5)\sqrt{(n-3)^2+1}+\sqrt{(n-3)^2+9}+\sqrt{(n-3)^2+4}+3\sqrt{13}+\sqrt{18}\\
		&=&\phi''(n,3)+(n-5)(\sqrt{(n-3)^2+1}-\sqrt{(n-2)^2+1})+2\sqrt{13}-\sqrt{10}\\
		&&+\sqrt{(n-3)^2+9}+\sqrt{(n-3)^2+4}-2\sqrt{(n-2)^2+9}-\sqrt{(n-2)^2+4}\\
		&<&\phi''(n,3)+g^{(1)}(3,2)-g^{(1)}(3,n-2)<\phi''(n,3).
	\end{eqnarray*}

	\vskip .2cm	
	
	{\bf Case 3.} $T \not\cong S_{n-1}$ and$T \not\cong S'_{n-1}$.
	
	\vskip .2cm
	
	In this case, by an argument similar to the proof of {Fact 1}, we have $T\cong S''_{n-1}$. Moreover, $v_nv_1\in E(G)$, which implies that the assertion holds for $k=1,n-1$. Let $v_2$ be the $3$-vertex in $T$, and let $v_3$ one $1$-vertex adjacent to $v_2$ in $T$. By an argument similar to the proofs of {Facts 2 and 3}, we have the following:
	
	\vskip .2cm
	
	{\bf Fact 4.} {\em (i)  If $k\ge 2$, then $v_nv_2\in E(G)$;
		
		(ii) If $k\ge 3$, then $v_nv_3\in E(G)$.}
	
	\vskip .2cm
	
	If $k=2$, then by Fact 4(i) and Proposition \ref{2} (i), \begin{eqnarray*}SO(G)&=&(n-5)\sqrt{(n-3)^2+1}+\sqrt{(n-3)^2+4}+\sqrt{(n-3)^2+16}+\sqrt{20}+2\sqrt{17}\\&<& \phi^*(n,2);\end{eqnarray*} and if $k=3$, then by Fact 4(ii) and Proposition \ref{2}(ii), $SO(G)=(n-5)\sqrt{(n-3)^2+1}+\sqrt{(n-3)^2+9}+\sqrt{(n-3)^2+16}+\sqrt{20}+\sqrt{17}+\sqrt{13}+5<\phi''(n,3)$.
	
	Therefore the proof of Theorem \ref{max3-2} is complete.
\end{proof}

Recall that $\mathscr{Q}(n,1)$ is the set of all trees of order $n$. By Theorems \ref{max}-\ref{max3-2}, we have the following results (see Table 1).

\begin{cor}
	Let $T$ be a tree of order $n$.
	
	(i) $SO(T)\le (n-1)\sqrt{(n-1)^2+1}$ with equality if and only if $T\cong S_n$ (see \cite{IGA8});
	
	(ii) If $T\not\cong S_n$, then $SO(T)\le (n-3)\sqrt{(n-2)^2+1}+\sqrt{(n-2)^2+4}+\sqrt{5}$ with equality if and only if $T\cong S'_n$ (see \cite{KCIG3}), where $n\ge 4$;
	
	(iii) If $T\not\cong S_n,S_n'$, then $SO(T)\le (n-4)\sqrt{(n-3)^2+1}+\sqrt{(n-3)^2+9}+2\sqrt{10}$ with equality if and only if $T\cong S''_n$, where $n\ge 6$.
\end{cor}

Let $\mathscr{U}_n$ be the set of all unicyclic graphs of order $n$. Then $\mathscr{U}_n\supset \mathscr{Q}(n,2)$.
Moreover, if $G\in  \mathscr{U}_n$ with the maximum $SO(G)$, then by Lemmas \ref{G'>G} and \ref{G*>G}, there is a $2$-vertex in the unique cycle of $G$, which implies $G\in \mathscr{Q}(n,2)$. Hence, by Theorems \ref{max}-\ref{max3-2}, we have the following results (see Table 1).

\begin{cor}
	Let $G$ be a unicyclic graph of order $n$ with $n\ge 3$. Then
	
	(i) $SO(G)\le \sqrt{8}+(n-3)\sqrt{(n-1)^2+1}+2\sqrt{(n-1)^2+4}$ with equality if and only if $G\cong Q_{n,2}$ (see \cite{RCJR1});
	
	(ii) If $G\not\cong Q_{n,2}$, then $SO(G)\le \sqrt{10}+\sqrt{13}+(n-4)\sqrt{(n-2)^2+1}+\sqrt{(n-2)^2+4}+\sqrt{(n-2)^2+9}$ with equality if and only if $T\cong Q'_{n,2}$, where $n\ge 5$;
	
	(iii) If $G\not\cong Q_{n,2},Q'_{n,2}$, then $SO(G)\le (n-4)\sqrt{(n-2)^2+1}+2\sqrt{(n-2)^2+4}+2\sqrt{8}$ with equality if and only if $T\cong Q^*_{n,2}$, where $n\ge 5$.
\end{cor}

\begin{table}[!hbp]\label{table1}
	\caption{Somber indices of quasi-tree graphs in $\mathscr{Q}(n,k)$}	
	\centering
	\vskip.2cm
	\begin{tabular}{|c|c|c|c|}
		\hline
		& the maximum & the second maximum & the third maximum\\
		\hline
		$k=1$  & 
		\setlength{\unitlength}{1mm}
		\begin{picture}(25,28)
			\thinlines
			\put(0,15){\circle*{1.5}}
			\put(10,20){\circle*{1.5}}
			\put(10,10){\circle*{1.5}}
			\put(10,13.5){\circle*{0.8}}
			\put(10,15.5){\circle*{0.8}}
			\put(10,17.5){\circle*{0.8}}
			\put(9,14){$~\left\}\rule{0mm}{7mm}\right.$\liuhao{$n-1$} }
			\put(0,15){\line(2,-1){10}}
			\put(0,15){\line(2,1){10}}
			\put(0,2){$S_n$ (see [11])}	
		\end{picture}
		&\setlength{\unitlength}{1mm}
		\begin{picture}(25,25)
			\thinlines
			\put(-2,15){\circle*{1.5}}
			\put(4,15){\circle*{1.5}}
			\put(10,15){\circle*{1.5}}
			\put(16,20){\circle*{1.5}}
			\put(16,10){\circle*{1.5}}
			\put(16,13.5){\circle*{0.8}}
			\put(16,15.5){\circle*{0.8}}
			\put(16,17.5){\circle*{0.8}}
			\put(15,14){$~\left\}\rule{0mm}{7mm}\right.$\liuhao{$n-3$} }
			\put(-2,15){\line(1,0){6}}
			\put(4,15){\line(1,0){6}}
			\put(10,15){\line(1,-0.8){6}}
			\put(10,15){\line(1,0.8){6}}
			\put(2,2){$S'_n$ (see [7])}	
		\end{picture}
		&\setlength{\unitlength}{1mm}
		\begin{picture}(25,25)
			\thinlines
			\put(-2,10){\circle*{1.5}}
			\put(-2,20){\circle*{1.5}}
			\put(4,15){\circle*{1.5}}
			\put(10,15){\circle*{1.5}}
			\put(16,10){\circle*{1.5}}
			\put(16,20){\circle*{1.5}}
			\put(16,13.5){\circle*{0.8}}
			\put(16,15.5){\circle*{0.8}}
			\put(16,17.5){\circle*{0.8}}
			\put(15,14){$~\left\}\rule{0mm}{7mm}\right.$\liuhao{$n-4$} }
			\put(4,15){\line(-1,-0.8){6}}
			\put(4,15){\line(-1,0.8){6}}
			\put(4,15){\line(1,0){6}}
			\put(10,15){\line(1,-0.8){6}}
			\put(10,15){\line(1,0.8){6}}
			\put(6,2){$S''_n$}		
		\end{picture}\\
		\hline
		$k=2$ &\setlength{\unitlength}{1mm}
		\begin{picture}(25,25)
			\thinlines
			\put(0,10){\circle*{1.5}}
			\put(0,20){\circle*{1.5}}
			\put(7,15){\circle*{1.5}}
			\put(14,10){\circle*{1.5}}
			\put(14,20){\circle*{1.5}}
			\put(14,13.5){\circle*{0.8}}
			\put(14,15.5){\circle*{0.8}}
			\put(14,17.5){\circle*{0.8}}
			\put(13,14){$~\left\}\rule{0mm}{7mm}\right.$\liuhao{$n-3$} }
			\put(7,15){\line(-1.5,-1.1){7}}
			\put(7,15){\line(-1.5,1.1){7}}
			\put(7,15){\line(1.5,1.1){7}}
			\put(7,15){\line(1.5,-1.1){7}}
			\put(0,20){\line(0,-1){10}}
			
			\put(0,2){$Q_{n,2}$ (see [4])}	
			
		\end{picture}
		& \setlength{\unitlength}{1mm}
		\begin{picture}(25,25)
			\thinlines
			\put(-2,20){\circle*{1.5}}
			\put(5,10){\circle*{1.5}}
			\put(5,20){\circle*{1.5}}
			\put(11,15){\circle*{1.5}}
			\put(17,10){\circle*{1.5}}
			\put(17,20){\circle*{1.5}}
			\put(17,13.5){\circle*{0.8}}
			\put(17,15.5){\circle*{0.8}}
			\put(17,17.5){\circle*{0.8}}
			\put(16,14){$~\left\}\rule{0mm}{7mm}\right.$\liuhao{$n-4$} }
			\put(11,15){\line(-1.4,-1.1){6}}
			\put(11,15){\line(-1.4,1.1){6}}
			\put(11,15){\line(1.4,1.1){6}}
			\put(11,15){\line(1.4,-1.1){6}}
			\put(5,20){\line(0,-1){10}}
			\put(5,20){\line(-1,0){7}}\put(6,2){$Q'_{n,2}$}
		\end{picture}  &  \setlength{\unitlength}{1mm}
		\begin{picture}(25,28)
			\thinlines
			\put(-2,15){\circle*{1.5}}
			\put(4,10){\circle*{1.5}}
			\put(4,20){\circle*{1.5}}
			\put(10,15){\circle*{1.5}}
			\put(16,10){\circle*{1.5}}
			\put(16,20){\circle*{1.5}}
			\put(16,13.5){\circle*{0.8}}
			\put(16,15.5){\circle*{0.8}}
			\put(16,17.5){\circle*{0.8}}
			\put(15,14){$~\left\}\rule{0mm}{7mm}\right.$\liuhao{$n-4$} }
			\put(10,15){\line(-1.4,-1.1){6}}
			\put(10,15){\line(-1.4,1.1){6}}
			\put(10,15){\line(1.4,1.1){6}}
			\put(10,15){\line(1.4,-1.1){6}}
			\put(4,20){\line(-1.4,-1.1){6}}
			\put(4,10){\line(-1.4,1.1){6}}
			
			\put(6,2){$Q^*_{n,2}$}
		\end{picture} \\
		\hline
		$k=3$ &\setlength{\unitlength}{1mm}
		\begin{picture}(27,28)
			\thinlines
			\put(-2,15){\circle*{1.5}}
			\put(4,20){\circle*{1.5}}
			\put(4,10){\circle*{1.5}}
			\put(10,15){\circle*{1.5}}
			\put(16,20){\circle*{1.5}}
			\put(16,10){\circle*{1.5}}
			\put(16,13.5){\circle*{0.8}}
			\put(16,15.5){\circle*{0.8}}
			\put(16,17.5){\circle*{0.8}}
			\put(15,14){$~\left\}\rule{0mm}{7mm}\right.$\liuhao{$n-4$}}
			\put(10,15){\line(-1.4,-1.1){6}}
			\put(10,15){\line(-1.4,1.1){6}}
			\put(10,15){\line(1.4,1.1){6}}
			\put(10,15){\line(1.4,-1.1){6}}
			\put(4,20){\line(-1.4,-1.1){6}}
			\put(4,10){\line(-1.4,1.1){6}}
			\put(-2,15){\line(1,0){12}}

			\put(5,2){$Q_{n,3}$}
			
		\end{picture}   & \setlength{\unitlength}{1mm}
		\begin{picture}(25,25)
			\thinlines
			\put(-2,15){\circle*{1.5}}
			\put(4,20){\circle*{1.5}}
			\put(4,10){\circle*{1.5}}
			\put(10,15){\circle*{1.5}}
			\put(16,20){\circle*{1.5}}
			\put(16,10){\circle*{1.5}}
			\put(16,13.5){\circle*{0.8}}
			\put(16,15.5){\circle*{0.8}}
			\put(16,17.5){\circle*{0.8}}
			\put(15,14){$~\left\}\rule{0mm}{7mm}\right.$\liuhao{$n-4$}}
			\put(10,15){\line(-1.4,-1.1){6}}
			\put(10,15){\line(-1.4,1.1){6}}
			\put(10,15){\line(1.4,1.1){6}}
			\put(10,15){\line(1.4,-1.1){6}}
			\put(4,20){\line(-1.4,-1.1){6}}
			\put(4,10){\line(-1.4,1.1){6}}
			\put(4,20){\line(0,-1){10}}			
			\put(6,2){$Q'_{n,3}$}
		\end{picture}& \setlength{\unitlength}{1mm}
		\begin{picture}(28,25)
			\thinlines
			\put(-1,20){\circle*{1.5}}
			\put(0,15){\circle*{1.5}}
			\put(6,20){\circle*{1.5}}
			\put(6,10){\circle*{1.5}}
			\put(12.5,15){\circle*{1.5}}
			\put(18,20){\circle*{1.5}}
			\put(18,10){\circle*{1.5}}
			\put(18,13.5){\circle*{0.8}}
			\put(18,15.5){\circle*{0.8}}
			\put(18,17.5){\circle*{0.8}}
			\put(17,14){$~\left\}\rule{0mm}{7mm}\right.$\liuhao{$n-5$}}
			\put(12,15){\line(-1.5,-1.2){6}}
			\put(12,15){\line(-1.5,1.2){6}}
			\put(12,15){\line(1.5,1.2){6}}
			\put(12,15){\line(1.5,-1.2){6}}
			\put(6,20){\line(-1.5,-1.2){6}}
			\put(6,10){\line(-1.5,1.2){6}}
			\put(0,15){\line(1,0){12}}
			\put(6,20){\line(-1,0){7}}

			\put(8,2){$Q''_{n,3}$}
		\end{picture}\\
		\hline 	 $4\le k\le n-2$
		&$Q_{n,k}$ ($n\ge 23$)&$Q^*_{n,k}$ ($n\ge 23$)&$Q'_{n,k}$ ($n\ge 23$)\\
		\hline $k=n-1$ &$Q_{n,n-1}$ (see [7])&$Q'_{n,n-1}$&$Q''_{n,n-1}$\\
		\hline 	
	\end{tabular}
	
\end{table}

Note that if we add an edge $e$ to a connected graph $G$, then $SO(G+e) >SO(G)$. So we have the following:

\begin{lem}\label{-1}
	$SO(Q_{n,k+1})>SO(Q_{n,k})$ for $1\le k\le n-2$.
\end{lem}

By Lemma \ref{-1} and Theorems \ref{max}-\ref{max3-1}, we have the following results.

\begin{cor}
	Let $G$ be a quasi-tree graph of order $n\ge 7$. Then
	
	(i) $SO(G)\le \sqrt{2}(n-1)+2(n-2)\sqrt{(n-1)^2+4}$ with equality if and only if $G\cong Q_{n,n-1}$ (see \cite{KCIG3});
	
	(ii) If $G\not\cong Q_{n,n-1}$, then $SO(G)\le (n-3)\sqrt{(n-1)^2+4}+(n-4)\sqrt{(n-2)^2+4}+\sqrt{(n-1)^2+(n-2)^2}+\sqrt{(n-2)^2+9}+\sqrt{(n-1)^2+9}+\sqrt{13}$ with equality if and only if $G\cong Q'_{n,n-1}$;
	
	(iii) If $G\not\cong Q_{n,n-1},Q'_{n,n-1}$, then $SO(G)\le (n-3)\sqrt{(n-1)^2+4}+\sqrt{(n-1)^2+16}+\sqrt{(n-1)^2+(n-3)^2}+\sqrt{(n-3)^2+16}+(n-5)\sqrt{(n-3)^2+4}+2\sqrt{20}$ with equality if and only if $G\cong Q''_{n,n-1}$.
\end{cor}

\section*{Conclusions} 

In this paper, we determined the maximum, the second maximum and the third maximum Sombor indices of all quasi-tree graphs in $\mathscr{Q}(n,k)$, respectively, which generalized some known results. Moreover, we characterized their corresponding extremal graphs, respectively. 

\section*{Declaration of Competing Interest}
The authors declare that they have no known competing financial interests or personal relationships that could have appeared to influence the work reported in this paper.

\section*{Acknowledgments} 

The authors express their sincere thanks to the support by National Natural Science Foundation of China(NNSFC) under grant number 11971158.




\end{document}